\documentclass[a4paper,reqno]{amsart}
\usepackage{amscd, amssymb, pb-diagram, amsmath}
\usepackage{amsthm}
\usepackage[bookmarks=true]{hyperref} 
\usepackage{slashed}

\usepackage{tikz}

\usepackage{tikz-cd}
\usepackage[margin=3.7cm]{geometry}
\usepackage{color}
\usepackage{tensor}
\usetikzlibrary{arrows,matrix,backgrounds,decorations.markings}
\numberwithin{equation}{section}

\usepackage[all]{xy}

\DeclareFontFamily{OT1}{pzc}{}
\DeclareFontShape{OT1}{pzc}{m}{it}{<-> s * [1.2] pzcmi7t}{}
\DeclareMathAlphabet{\mathpzc}{OT1}{pzc}{m}{it}

\def\End{\operatorname{End}}

\def\ker{\operatorname{Ker}}

\def\dim{\operatorname{dim}}

\def\sup{\operatorname{sup}}

\def\Re{\operatorname{Re}}
\def\Im{\operatorname{Im}}

\def\C{\mathbb{C}}

\def\R{\mathbb{R}}
\def\N{\mathbb{N}}

\newtheorem{thm*}{Theorem}
\newtheorem{thm}{Theorem}[section]

\newtheorem{cor}[thm]{Corollary}

\newtheorem{prop}[thm]{Proposition}

\newtheorem{problem*}[thm*]{Problem}

\theoremstyle{definition}
\newtheorem{definition}[thm]{Definition}
\newtheorem{definition*}[thm*]{Definition}

\theoremstyle{remark}
\newtheorem{remark}[thm]{Remark}

\newtheorem{question}[thm]{Question}

\usepackage{xcolor}

\begin{document}

\date{\today}

\title{The index of sub-laplacians: beyond contact manifolds}

\author{Magnus Goffeng, Bernard Helffer}

\address[M. Goffeng]{Centre for Mathematical Sciences, Lund University, Box 118, 221 00 LUND, Sweden}
\email{magnus.goffeng@math.lth.se}

\address[B. Helffer]{Laboratoire de Math\'ematiques Jean Leray, 
Nantes Universit\'e,
44 000 Nantes - FRANCE }
\email{bernard.helffer@univ-nantes.fr}

\maketitle

\begin{abstract}
In this paper we study the following question: do sub-Laplacian type operators have non-trivial index theory on Carnot manifolds in higher degree of nilpotency? The problem relates to characterizing the structure of the space of hypoelliptic sub-Laplacian type operators, and results going back to Rothschild-Stein and Helffer-Nourrigat. In two degrees of nilpotency, there is a rich index theory by work of van Erp-Baum on contact manifolds, that was later extended to polycontact manifolds by Goffeng-Kuzmin. We provide a plethora of examples in higher degree of nilpotency where the index theory is trivial. 
\end{abstract}

\section{Introduction}
We consider a Carnot manifold $\mathfrak{X}$ with degree of nilpotency $r$ coming from an equiregular differential system. 
For further geometric context of Carnot manifolds defined from equiregular differential system, see \cite{gromovcc,mithcjdg}, and for further analytic context see \cite{AMY,deverd,metivier}. 
Carnot manifolds are also known as filtered manifolds or equiregular sub-Riemannian manifolds.
In other words, we have chosen a subbundle $H\subseteq T\mathfrak{X}$ that induces a filtering 
\begin{equation}
\label{filteredeq}
T\mathfrak{X}=T^{-r}\mathfrak{X}\supsetneq T^{-r+1}\mathfrak{X}\supsetneq \ldots\supsetneq T^{-2}\mathfrak{X}\supsetneq T^{-1}\mathfrak{X}\supsetneq 0,
\end{equation}
of sub-bundles such that $T^{-1}\mathfrak{X}=H$ and $T^{j-1}\mathfrak{X}=T^j\mathfrak{X}+[T^{-1}X,T^j\mathfrak{X}]$ for any $j$. The associated graded bundle $\mathfrak{t}_H\mathfrak{X}:=\oplus_j T^{-j}\mathfrak{X}/T^{-j+1}\mathfrak{X}$ carries a fibrewise Lie algebra structure induced from commutators of vector fields. While $T^{-1}\mathfrak{X}=H$ is bracket generating, for any $x\in \mathfrak{X}$, the nilpotent Lie algebra $\mathfrak{t}_H\mathfrak{X}_x$ is generated by $T^{-1}_x\mathfrak{X}$ and has nilpotency degree $r$. We write $\mathcal{DO}_H^m(\mathfrak{X};E)$ for the space of differential operators acting on a vector bundle $E\to \mathfrak{X}$ of Heisenberg order at most $m$ relative to the filtering. 

Among the differential operators in $\mathcal{DO}_H^m(\mathfrak{X};E)$, the $H$-elliptic operators form a distinguished class that define Fredholm operators on an appropriate scale of Heisenberg-Sobolev spaces. We recall the notion of $H$-ellipticity and the closely related Rockland condition two paragraphs below. By abstract deformation arguments \cite{baumvanerp,mohsen2} one finds as rich index theory for $H$-elliptic operators in the Heisenberg calculus as one does in the ordinary pseudodifferential calculus -- yet no concrete examples with non-trivial index theory are known in general. A natural quest is to look for classes of differential operators that are $H$-elliptic and carry non-trivial index theory. A further reaching goal is to find differential operators capturing the geometry of Carnot manifolds in the spirit that Dirac operators do for Riemannian geometry and classical elliptic operators \cite{AS,BD,LM}.

In light of recent results \cite{baumvanerp,goffkuz}, we approach this line of questioning by studying operators $D_\gamma\in \mathcal{DO}_H^2(\mathfrak{X};E)$ of the form 
\begin{equation}
\label{jnkjnkjnad}
D_\gamma=\sum_{j=1}^{n} X_j^2+\sum_{l=1}^{m}\gamma_l Y_l+\mathcal{DO}_H^1(\mathfrak{X};E),
\end{equation}
where $X_1,\ldots ,X_{n}$ span $T^{-1}\mathfrak{X}$ in all points and $Y_1,\ldots, Y_{m}$ span $T^{-2}\mathfrak{X}/T^{-1}\mathfrak{X}$ in all points. We call such an operator a \emph{sub-Laplacian} and note that $D_\gamma$ is uniquely determined modulo $\mathcal{DO}_H^1(\mathfrak{X};E)$ from the classical principal symbol $\sigma^2(\sum_{j=1}^{n} X_j^2)$ and the restricted sub-principal symbol
\begin{equation}
\label{lknlnblabla}
\gamma:=\sigma^1\left(D_\gamma-\sum_{j=1}^{n} X_j^2\right)\bigg|_{(T^{-1}\mathfrak{X})^\perp/(T^{-2}\mathfrak{X})^\perp}.
\end{equation}
The possible classical principal symbols $\sigma^2(\sum_{j=1}^{n} X_j^2)$ stand in a one-to-one correspondence with metrics on the sub-bundle $H=T^{-1}\mathfrak{X}\subseteq T\mathfrak{X}$. 

Let us briefly review the Heisenberg calculus \cite{davehaller,goffkuz,melinoldpreprint,vanerpyuncken}, see also \cite{bealsgreiner,pongemono,taylor} for the case of contact manifolds or more generally when $T\mathfrak{X}/H$  has rank $1$. Of particular importance is the Rockland condition and related notions of $H$-ellipticity that are later used to conclude analytic properties of sub-Laplacians. We write $\mathfrak{t}_H\mathfrak{X}:=\oplus_j T^j\mathfrak{X}/T^{j+1}\mathfrak{X}$ for the graded bundle associated with the filtration \eqref{filteredeq} and view $\mathfrak{t}_H\mathfrak{X}$ as a bundle of graded Lie algebras with fibrewise bracket induced from the Lie bracket of vector fields. We can fibrewise integrate the Lie algebra structure to a bundle $T_H\mathfrak{X}\to \mathfrak{X}$ of simply connected, nilpotent Lie groups. By declaring a vector field with values in $T^{-j}\mathfrak{X}$ to have order $\leq j$, we can for a natural number $m$ define the space of differential operators $\mathcal{DO}_H^m(\mathfrak{X};E)$ of Heisenberg order $m$ on a vector bundle $E\to \mathfrak{X}$. By taking the principal part with respect to the Heisenberg order, we arrive at a symbol mapping 
$$\sigma^m_H:\mathcal{DO}_H^m(\mathfrak{X};E)/\mathcal{DO}_H^{m-1}(\mathfrak{X};E)\xrightarrow{\sim} C^\infty(\mathfrak{X},\mathcal{U}_m(\mathfrak{t}_H\mathfrak{X})),$$
where $\mathcal{U}_m$ denotes the degree $m$ part of the universal enveloping Lie algebra. A key feature is that the Lie algebraic machinery is compatible with the operators in the sense that $\sigma^{m_1+m_2}_H(D_1D_2)=\sigma^{m_1}_H(D_1)\sigma^{m_2}_H(D_2)$ when $D_1$ and $D_2$ are of order $m_1$ and $m_2$, respectively. An important notion for a differential operator $D\in \mathcal{DO}_H^m(\mathfrak{X};E)$ is the Rockland condition, which is said to be satisfied if for any $x\in \mathfrak{X}$ and any non-trivial, irreducible, unitary representation $\pi$ of the simply connected Lie group $T_H\mathfrak{X}_x$ integrating $\mathfrak{t}_H\mathfrak{X}_x$, the represented symbol $\pi(\sigma_H^m(D)_x)$ is injective on the space of smooth vectors $C^\infty(\pi)$. For the operators $D_\gamma$ of interest in this paper, the Rockland condition was proven to be equivalent to maximal hypoellipticity in \cite{rothcrit},  with more general results found later \cite{AMY,davehaller}. For the equiregular differential systems we consider, we can define $H$-ellipticity of a differential operator $D\in \mathcal{DO}_H^m(\mathfrak{X};E)$ as $D$ admitting an inverse modulo lower order terms in the Heisenberg calculus, or equivalently that $\sigma^m_H(D)$ admits an inverse in the Heisenberg symbol algebra \cite{davehaller,goffkuz}, that by \cite{davehaller} is equivalent to $D$ and $D^*$ satisfying the Rockland condition. The problem we will study is the following. 

\begin{problem*}
\label{mainprob}
For a fixed metric on $H=T^{-1}\mathfrak{X}$, which first order polynomial sections 
$$\gamma:(T^{-1}\mathfrak{X})^\perp/(T^{-2}\mathfrak{X})^\perp\to \End(E),$$
give rise to an $H$-elliptic operator $D_\gamma$ (as in \eqref{jnkjnkjnad} and \eqref{lknlnblabla}) with non-trivial Fredholm index? 
\end{problem*}

 Problem \ref{mainprob} is completely understood in the case of contact manifolds in work by Baum-van Erp \cite{baumvanerp} that followed a substantial effort from Epstein-Melrose \cite{melroseeptein}. Similar results hold for polycontact manifolds \cite{goffkuz}. Contact- and polycontact manifolds are Carnot manifolds of nilpotency degree $r=2$ with a strong non-degeneracy condition on the Levi bracket $H\times H\to T\mathfrak{X}/H$. For Carnot manifolds of higher nilpotency degree close to nothing is known about non-triviality of the index. The second part of Problem \ref{mainprob}, that of computing the index, is by the methods of \cite{baumvanerp,goffkuz} closely connected to the first part of Problem \ref{mainprob}, that of characterizing $H$-ellipticity. Since $H$-ellipticity is equivalent to the Rockland condition, the problem of characterizing $H$-ellipticity reduces to pointwise properties. More precisely, given a graded nilpotent Lie algebra $\mathfrak{g}=\mathfrak{g}_{-1}\oplus \mathfrak{g}_{-2}\oplus\cdots\oplus \mathfrak{g}_{-r}$, generated by $\mathfrak{g}_{-1}$, and $N\in \N_{>0}$ we consider the set 
$$\mathcal{E}(\mathfrak{g},N):=\{\gamma=(\gamma_l)_{l=1}^m\in M_N(\C)^m: \mbox{the operator \eqref{jnkjnkjnad} is hypoelliptic on $G$}\}.$$
Here $m=\dim(\mathfrak{g}_{-2})$ and we have implicitly chosen an inner product structure on $\mathfrak{g}_{-1}$ that fixes the principal symbol $\sigma^2(\sum_{j=1}^{n} X_j^2)$ and $(Y_l)_{l=1}^m$ in \eqref{jnkjnkjnad} as an orthonormal basis of $\mathfrak{g}_{-2}$. Also, $N$ encodes the rank of the vector bundle $E$. \emph{The reader should beware that we can, and will, identify $\mathcal{E}(\mathfrak{g},N)$ with a subset of $\mathfrak{g}_{-2}\otimes M_N(\C)$ as well as with a subset of linear polynomials $\mathfrak{g}_{-2}^*\to M_N(\C)$.  To ensure compatibility with \eqref{lknlnblabla}, this identification is made by defining 
$$\gamma(\xi):=i\xi(\gamma), \quad \xi\in \mathfrak{g}_{-2}^*,$$
where $\xi(\gamma)\in M_N(\C)$ is defined from contracting $\gamma\in \mathfrak{g}_{-2}\otimes M_N(\C)$ with $\xi\in \mathfrak{g}_{-2}^*$ in the first leg.} By construction it is clear that a first order polynomial section $\gamma:(T^{-1}\mathfrak{X})^\perp/(T^{-2}\mathfrak{X})^\perp\to \End(E)$ on a compact Carnot manifold $\mathfrak{X}$ defines an $H$-elliptic operator $D_\gamma$ (as in \eqref{jnkjnkjnad} and \eqref{lknlnblabla}) if and only if $\gamma_x, -\gamma_x^*\in \mathcal{E}(\mathfrak{t}_H\mathfrak{X}_x,\mathrm{rk}(E_x))$ for all $x\in \mathfrak{X}$. For contact, or more generally polycontact, manifolds $\mathcal{E}(\mathfrak{g},N)$ consists precisely of those $\gamma \in \mathfrak{g}_{-2}\otimes M_N(\C)$ such that $\gamma(\xi)$ has spectrum outside $2\N+\mathrm{rk}(H)/2$ for $\xi\in \mathfrak{g}_{-2}^*$ being unit length in the inner product induced from that on $\mathfrak{g}_{-1}$ and the Lie bracket. 

\begin{definition*}
Let $\mathfrak{g}$ be a graded nilpotent Lie algebra such that $\mathfrak{g}_{-1}$ has an inner product and generates $\mathfrak{g}$. For $N\in \N_{>0}$ we say that {\bf $\mathfrak{g}$ has property $\mathcal{E}^*N$} if for any $t\in (0,1]$ we have that 
$$t\mathcal{E}(\mathfrak{g},N)\subseteq\mathcal{E}(\mathfrak{g},N).$$
\end{definition*}

The importance of property $\mathcal{E}^*N$ is found in Theorem \ref{indexzerofromegn} showing that if $\mathfrak{X}$ is a compact Carnot manifold with $\mathfrak{t}_H\mathfrak{X}_x$ having property $\mathcal{E}^*N$ for all $x$, then Problem \ref{mainprob} admits no solution $\gamma$ whenever $E$ has rank $N$.

Throughout the paper we prove property $\mathcal{E}^*N$ in various cases. In Corollary \ref{cortors}, we carefully apply Rothschild-Stein's work \cite{rothstein} to show that if $\mathfrak{g}_{-2}$ has dimension different from $2$ then property $\mathcal{E}^*1$ holds. In Proposition \ref{alknalkdnad} we modify a work of Helffer \cite{helffercomp} to show that if $\mathfrak{g}$ has nilpotency degree $>2$ and $\mathfrak{g}_{-2}$ is one-dimensional then property $\mathcal{E}^*N$ holds for all $N$. In Proposition \ref{alakdnjak} we show that for the nilpotent Lie algebra $\mathfrak{n}(4)$ of strictly upper triangular $4\times 4$ matrices, property $\mathcal{E}^*N$ holds for all $N$. Invoking Theorem \ref{indexzerofromegn} as discussed above we can deduce the next theorem.

\begin{thm*}
\label{mainthm}
Let $\mathfrak{X}$ denote a compact Carnot manifold defined from an equiregular differential system as above, and write $r$ for its nilpotency degree. The operators $D_\gamma\in \mathcal{DO}_H^2(\mathfrak{X};E)$ satisfy that if $D_\gamma$ is $H$-elliptic then
$$\mathrm{index}(D_\gamma)=0,$$
under any of the following geometric assumptions on $(\mathfrak{X},E)$:
\begin{enumerate}
\item $E$ is a line bundle and for all $x\in \mathfrak{X}$, $\mathrm{rank}(T^{-2}\mathfrak{X}/T^{-1}\mathfrak{X})\neq 2$ .
\item The nilpotency degree of $\mathfrak{X}$ satisfies $r>2$ and $T^{-2}\mathfrak{X}/T^{-1}\mathfrak{X}$ is a line bundle.
\item $\mathfrak{X}$ is a regular parabolic geometry of type $(SL_4(\R),P)$, where $P$ is the minimal parabolic subgroup.
\end{enumerate}
In fact, in all cases above it even holds that $[D_\gamma]=0$ in $K_0(\mathfrak{X})$ whenever $D_\gamma$ is $H$-elliptic.
\end{thm*}

Here $K_0(\mathfrak{X})$ denote the even $K$-homology of the topological space $\mathfrak{X}$, for details see \cite{BD,HigRoe}.

\begin{remark}
Instead of an equiregular differential system, we can more generally consider a finitely generated submodule $\mathpzc{E}\subseteq C^\infty(\mathfrak{X},T\mathfrak{X})$ that satisfies the Hörmander condition, i.e. $\mathpzc{E}$ generates $C^\infty(\mathfrak{X},T\mathfrak{X})$ as a Lie algebra. We have hope that the ideas underlying Theorem \ref{mainthm} extend also to this setting following \cite{AMY}.  A related example can be found in \cite{Mohsen3}. Indeed, when the system is not equiregular but satisfies the H\"ormander condition of degree $r$, the Rockland condition that in the equiregular case is posed for all $x\in \mathfrak{X}$ and all non-trivial, irreducible, unitary representations of $T_H\mathfrak{X}_x$ can be replaced by a condition that over $x\in \mathfrak{X}$ is posed for all non-trivial representations from a conical closed set $\Gamma_x$ in $\widehat G_{p,r}$ where $G_{p,r}$ is the free group whose Lie algebra $\mathfrak{g}_{r,p}$ is the maximal graded algebra of nilpotency degree $r$ with $p$ generators of degree $-1$. This was conjectured by Helffer-Nourrigat in 1979 (see \cite{HNo}) and recently proved in \cite{AMY}.
\end{remark}

Theorem \ref{mainthm} shows that it is commonly occurring for Problem \ref{mainprob} to admit no solution, however it is unsatisfactory in the sense that all conditions rely on some structure being low dimensional. In the hope of providing a stronger result and understanding for which Carnot manifolds of higher nilpotency degree that Problem \ref{mainprob} admits a solution, we look to the work of Rothschild-Stein \cite{rothstein}. We write $\|\cdot \|_{W}$ for the norm on $\mathfrak{g}_{-2}\otimes M_N(\R)$ defined from realizing $\mathfrak{g}_{-2}\otimes M_N(\R)$ as a quotient of $\mathfrak{so}_n(M_N(\R))$, for $n=\dim(\mathfrak{g}_{-1})$. More details can be found in Subsection \ref{somelei}. This norm was studied in the literature \cite{helffercomp,rothstein} for $N=1$. We can write $\gamma\in \mathfrak{g}_{-2}\otimes M_N(\C)$ as $\gamma=a+ib$ for $a$ and $b$ self-adjoint, and we say that $\gamma$ has {\bf property $\alpha$} if for  any unit vectors $\xi\in \mathfrak{g}_{-2}^*$ and $v\in \ker(\xi(a))$ we have that $|\langle v, \xi(b)v\rangle_{\C^N}|<1$. Here $\xi(\gamma)\in M_N(\C)$ is defined from $\xi\in \mathfrak{g}_{-2}^*$ with $\gamma\in \mathfrak{g}_{-2}\otimes M_N(\C)$ in the first leg,  so $\xi(\gamma)=-i\gamma(\xi)$. Following the ideas in \cite{rothstein}, we prove the following result in Subsection \ref{extrs} that sandwiches $\mathcal{E}(\mathfrak{g},N)$ between two sets that are starshaped with respect to $0$ in the spirit of property $\mathcal{E}^*N$.

\begin{thm*}
Let $\mathfrak{g}$ be a graded nilpotent Lie algebra generated by $\mathfrak{g}_{-1}$. It holds that
$$\mathcal{E}(\mathfrak{g},N)\supseteq \{\gamma\in M_N(\C)^m: \|\Im(\gamma)\|_{W/\R\Re(\gamma))}<1\}.$$
Moreover, if $\mathfrak{g}$ has the property that the Kirillov form $\omega_\xi(X,Y):=\xi([X,Y])$ is degenerate on $\mathfrak{g}_{-1}$ for any $\xi\in \mathfrak{g}_{-2}^*$, then 
$$\mathcal{E}(\mathfrak{g},N)\subseteq \{\gamma\in M_N(\C)^m: \; \gamma\; \mbox{has property $\alpha$}\}.$$
\end{thm*}

The paper is organized as follows. In Section \ref{secadkljnakjlan} we provide some preliminary material and a proof that property $\mathcal{E}^*N$ implies vanishing of the index of sub-Laplacians. We study instances when property $\mathcal{E}^*N$ holds in Section \ref{seconestan}. We study several examples showcasing that Problem \ref{mainprob} often has no solutions in Section \ref{ljnljnad}. Finally, in Section \ref{fillingap} we provide the final details in the proof of Rothschild-Stein's result \cite[Theorem 2]{rothstein} covering a case that was overlooked.

\subsection*{Acknowledgements} 

The authors wish to thank the anonymous referees for their helpful suggestions. The first listed author was supported by the Swedish Research Council Grant VR 2018-0350.

\section{Preliminaries}
\label{secadkljnakjlan}

\subsection{Some Lie algebra}
\label{somelei}
We consider graded nilpotent Lie algebras 
$$\mathfrak{g}=\bigoplus_{j=1}^r \mathfrak{g}_{-j},$$
such that $\mathfrak{g}_{-1}$ generates $\mathfrak{g}$ as a Lie algebra. Such Lie algebras are also known as stratified Lie algebras. In particular, $r$ coincides with the nilpotency degree of $\mathfrak{g}$  and $\mathfrak{g}_{-j}=\sum_{l=1}^{j-1}[\mathfrak{g}_{-l},\mathfrak{g}_{-j+l}]$ for any $j>1$.

We will study the $H$-ellipticity of the operators $D_\gamma$ on a Carnot manifolds by means of their Heisenberg principal symbol, which in other words mean that we fix a point $x\in \mathfrak{X}$ and consider the translation invariant differential operator on the simply connected nilpotent Lie group $G$ integrating $\mathfrak{g}:=\mathfrak{t}_H\mathfrak{X}_x$ that takes the form
$$D_\gamma=\sum_{j=1}^{n} X_j^2+\sum_{l=1}^{m}\gamma_{l} Y_l,$$
where $X_1,\ldots, X_n$ is the basis of a subspace $\mathfrak{g}_{-1}\subseteq \mathfrak{g}$ generating the nilpotent Lie algebra $\mathfrak{g}$ and $Y_1,\ldots, Y_m$ is a basis for the subspace $\mathfrak{g}_{-2}\subseteq \mathfrak{g}$. We tacitly assume that $\mathfrak{g}_{-1}$ is equipped with an inner product in which $X_1,\ldots, X_n$ is an orthonormal basis and that $Y_1,\ldots, Y_m$ is an orthonormal basis for $\mathfrak{g}_{-2}$ is equipped with the inner product induced from the projection mapping $\wedge^2\mathfrak{g}_{-1}\to \mathfrak{g}_{-2}$ that the Lie bracket defines. Here $(\gamma_l)_{l=1}^m\subseteq M_N(\C)$ is some collection of matrices. 

We remark that alternatively, we can also describe our operator in the form 
\begin{equation}
\label{blabla}
D_\gamma=\sum_{j=1}^{n} X_j^2+\sum_{k,l=1}^{m}\gamma_{k,l} [X_k,X_l],
\end{equation}
where $(\gamma_{k,l})_{k,l=1}^n\subseteq M_N(\C)$ is some collection of matrices. 
Write $\mathfrak{so}_n$ for the algebraic Lie algebra of anti-symmetric $n\times n$-matrices, and $\mathfrak{so}_n(M_N(\C))$ for the anti-symmetric matrices over $M_N(\C)$. That is, $(\gamma_{k,l})_{k,l=1}^n\in \mathfrak{so}_n(M_N(\C))$ if and only if $\gamma_{k,l}=-\gamma_{l,k}$. We can assume that the collection $(\gamma_{k,l})_{k,l=1}^n$ has been choosen in $\mathfrak{so}_n(M_N(\C))$. The collection $(\gamma_{k,l})_{k,l=1}^n\in \mathfrak{so}_n(M_N(\C))$ is uniquely determined from $(\gamma_l)_{l=1}^m$ as above modulo the space 
$$\mathpzc{S}_{\C,N}:=\left\{(s_{k,l})_{k,l=1}^n\in \mathfrak{so}_n(M_N(\C)):\sum_{l,k=1}^{n}s_{l,k} [X_l,X_k]=0\right\}.$$
Note that $\mathpzc{S}_{\C,N}=M_N(\C)\otimes_\C \mathpzc{S}_{\C,1}$. We write $\mathpzc{S}$ for the real part of $\mathpzc{S}_{\C,1}$.

\begin{prop}
\label{knlknada}
There is a linear isomorphism 
$$\delta:M_N(\C)^m\to \mathfrak{so}_n(M_N(\C))/\mathpzc{S}_{\C,N},$$
defined from 
$$\sum_{k,l=1}^{n}\delta(\gamma)_{k,l} [X_k,X_l]=\sum_{l=1}^{m}\gamma_{l} Y_l.$$
The map $\delta$ respects entrywise complex conjugation and hermitean conjugates, where hermitean conjugate in $\mathfrak{so}_n(M_N(\C))$ is entry-wise defined by $((\gamma_{k,l})_{k,l=1}^n)^*:=(\gamma_{k,l}^*)_{k,l=1}^n$. 
\end{prop}

\begin{proof}
Since $(Y_l)_{l=1}^m$ is a basis for $\mathfrak{g}_{-2}$, there exists a matrix $(\gamma_{k,l})_{k,l=1}^n\in \mathfrak{so}_n(M_N(\C))$ with $\sum_{k,l=1}^{m}\gamma_{k,l} [X_k,X_l]=\sum_{l=1}^{m}\gamma_{l} Y_l$ and it is clear that $\delta(\gamma):=[(\gamma_{k,l})_{k,l=1}^n]\in \mathfrak{so}_n(M_N(\C))/\mathpzc{S}_{\C,N}$ is well defined. Again since $(Y_l)_{l=1}^m$ is a basis for $\mathfrak{g}_{-2}$, the map $\delta$ is an isomorphism. It follows from the construction that $\delta$ respects entrywise complex conjugation and hermitean conjugates. 
\end{proof}

\begin{remark}
\label{inclusionrem}
The two step nilpotent Lie algebra $\mathfrak{g}_{-1}\oplus \mathfrak{g}_{-2}\equiv \mathfrak{g}/\oplus_{j=3}^r \mathfrak{g}_{-j}$ is by Eberlein's construction \cite{eberlein} of the following form. We set $V:=\mathfrak{g}_{-1}$ which is an inner product space. We identify $\mathfrak{so}(V)$ with the antisymmetric two-forms. There is an injective map $\iota:\mathfrak{g}_{-2}\to \mathfrak{so}(V)$ defined from $\iota(Y)X:=\mathrm{Ad}^*(X)(Y)$ for $Y\in \mathfrak{g}_{-2}$ and $X\in V=\mathfrak{g}_{-1}$, i.e. the $2$-form $\iota(Y)$ is determined from $\iota(Y)(X,X')=(Y,[X,X'])$. Write $W:=\iota(\mathfrak{g}_{-2})$. For $\omega\in V^*\wedge V^*=\mathfrak{so}(V)$ we write $\omega_W\in W$ for the the orthogonal projection onto $W$. The space $V_W:=V\oplus W$ is a Lie algebra in the Lie bracket $[(X,Y),(X',Y')]:=(0,(X\wedge X')_W)$ and the canonical linear isomorphism $\mathfrak{g}/\oplus_{j=3}^r \mathfrak{g}_{-j}\xrightarrow{\sim}V_W$ is also a Lie algebra isomorphism. It is clear from a dimensional consideration that the composition 
$$\mathfrak{g}_{-2}\xrightarrow{\iota} \mathfrak{so}(V)\to \mathfrak{so}(V)/\mathpzc{S},$$
is a linear isomorphism.

The isomorphism $\delta$ from Proposition \ref{knlknada} can be factorized as 
\begin{align*}
M_N(\C)^m=M_N(\C)\otimes_\R \mathfrak{g}_{-2}&\xrightarrow{\mathrm{id}\otimes \iota} M_N(\C)\otimes_\R\mathfrak{so}(V)\to\\
&\to M_N(\C)\otimes_\R(\mathfrak{so}(V)/\mathpzc{S})=\mathfrak{so}_n(M_N(\C))/\mathpzc{S}_{\C,N}.
\end{align*}
The isomorphism $\delta$ from Proposition \ref{knlknada} is therefore isometric if we view $\mathfrak{g}_{-2}\cong \R^m$ as the matrix normed space induced from the inclusion $\iota$.
\end{remark}

\begin{remark}
Fixing our graded nilpotent Lie algebra $\mathfrak{g}$ as above, and the basis for $\mathfrak{g}_{-1}\oplus \mathfrak{g}_{-2}$, we can with any $\gamma\in M_N(\C)^m$ associate a first order polynomial on $\mathfrak{g}_{-2}^*$ that we by an abuse of notation write 
$$\gamma:\mathfrak{g}_{-2}^*\to M_N(\C), \quad \gamma(\xi):=i\sum_{l=1}^m \gamma_l \xi_l.$$
This identification comes from viewing $\mathfrak{g}_{-2}\otimes M_N(\C)=M_N(\C)^m$ as the subspace of first order homogeneous polynomials on $\mathfrak{g}_{-2}^*$ with values in $M_N(\C)$. Our convention ensures that self-adjoint elements of $\mathfrak{g}_{-2}\otimes M_N(\C)$, with $*$-operation $(\sum_l Y_l\otimes \gamma_l)^*=-\sum_l Y_l\otimes \gamma_l$, correspond to first order homogeneous polynomials taking self-adjoint values. 
\end{remark}

\subsection{A characterizing set for the Rockland condition}

\begin{definition}
Define the subset $\mathcal{E}(\mathfrak{g},N)\subseteq M_N(\C)^m$ to consist of matrices $\gamma=(\gamma_{l})_{l=1}^m$ such that the represented operator $\pi(D_\gamma)$ is injective in any non-trivial irreducible unitary representation $\pi$ of $G$.

More generally, for a closed, conical subset $\Phi\subseteq \widehat{G}$ we write $\mathcal{E}_\Phi(\mathfrak{g},N)\subseteq M_N(\C)^m$ for the subset of matrices $\gamma=(\gamma_{l})_{l=1}^m$ such that the represented operator $\pi(D_\gamma)$ is injective for all $\pi\in \Phi \setminus \{1\} $.
\end{definition}
 
 The extension of the Rockland condition to closed cones has been studied by Nourrigat \cite{nourrigat87} and Hebisch \cite{hebisch98}.

\begin{remark}
Note that $D_\gamma$ is hypoelliptic if and only if $\gamma \in \mathcal{E}(\mathfrak{g},N)$ and $D_\gamma$ is H-elliptic if and only if $\gamma \in \mathcal{E}(\mathfrak{g},N)\cap (-\mathcal{E}(\mathfrak{g},N)^*)$ where $\gamma^*:=(\gamma_{l}^*)_{l=1}^m$ is the total hermitean adjoint. We also note that by homogeneity, whenever $\Phi_0\subseteq \Phi\setminus\{1\}$ is a transversal for the dilation action on $\Phi\setminus\{1\}$ we have that $\mathcal{E}_\Phi(\mathfrak{g},N)$ coincides with the set of all matrices $\gamma=(\gamma_{l})_{l=1}^m$ such that the represented operator $\pi(D_\gamma)$ is injective for all $\pi\in \Phi_0 $.
\end{remark}

\begin{prop}
\label{somsplpklad}
Let $\mathfrak{g}$ be a graded nilpotent Lie algebra such that $\mathfrak{g}_{-1}$ generates $\mathfrak{g}$ and is equipped with an inner product.
\begin{enumerate}
\item If $\Phi\subseteq \Phi'$ then $\mathcal{E}_{\Phi'}(\mathfrak{g},N)\subseteq \mathcal{E}_\Phi(\mathfrak{g},N)$.
\item The subset 
$$\mathcal{E}_\Phi(\mathfrak{g},N)\subseteq M_N(\C)^m,$$ 
is open for any closed, conical subset $\Phi\subseteq \widehat{G}\setminus \{1\}$. 
\item $0\in \mathcal{E}(\mathfrak{g},N)$.
\end{enumerate}
\end{prop}

\begin{proof}
The first item is obvious. The Rockland condition relative to $\Phi$ for a homogeneous element $P$ of the enveloping algebra ensures maximal estimates for the operators $\pi(P)$ with uniform constant with respect to $\pi \in \Phi\setminus\{1\}$, see \cite{nourrigat87}. Now it follows from a continuity argument that $\mathcal{E}_\Phi(\mathfrak{g},N)\subseteq M_N(\C)^m$ is open, proving the second item. We have that $0\in \mathcal{E}(\mathfrak{g},N)$ by H\"ormander's sum of squares theorem \cite{horsumsquares}.
\end{proof}

\begin{prop}
\label{lknjnada}
Suppose that $\phi: \mathfrak{g}\to \mathfrak{g}'$ is a surjective graded Lie algebra homomorphism which induces an isometry $\mathfrak{g}_{-1}/(\ker\phi\cap \mathfrak{g}_{-1})\to \mathfrak{g}_{-1}'$. Set $m:=\dim(\mathfrak{g}_{-2})$ and $m':=\dim(\mathfrak{g}'_{-2})$. Then there is a mapping 
$$\phi_*:M_N(\C)^m\to M_N(\C)^{m'},$$
defined from 
$$\mathcal{U}(\phi)(D_\gamma)=D_{\phi_*(\gamma)},$$
where $\mathcal{U}(\phi):\mathcal{U}(\mathfrak{g})\to \mathcal{U}(\mathfrak{g}')$ is the functorial map between universal enveloping algebras. 

Moreover $\phi_*:\mathcal{E}(\mathfrak{g},N)\to \mathcal{E}(\mathfrak{g}',N)$ is a well defined mapping such that $\phi_*(\gamma_1)=\phi_*(\gamma_2)$ if and only if 
$$\sum_{l=1}^m (\gamma_{1,l}-\gamma_{2,l})Y_l\in \ker\phi.$$
The map $\phi_*:\mathcal{E}(\mathfrak{g},N)\to \mathcal{E}(\mathfrak{g}',N)$ extends to a well defined surjection $\phi_*:\mathcal{E}_\Phi(\mathfrak{g},N)\to \mathcal{E}(\mathfrak{g}',N)$ where $\Phi=\hat{\phi} (\widehat{G'})\setminus \{1\}\subseteq \widehat{G}\setminus \{1\}$ and $\hat{\phi}:\widehat{G'}\to \widehat{G}$ denotes the induced map on the spectrum.
\end{prop}

\begin{proof}
Since $\phi$ induces an isometry $\mathfrak{g}_{-1}/(\ker\phi\cap \mathfrak{g}_{-1})\to \mathfrak{g}_{-1}'$, surjectivity of $\phi$ implies that 
$$\mathcal{U}(\phi)(D_\gamma)=\mathcal{U}(\phi)\left(\sum_j X_j^2+\sum_l \gamma_lY_l\right)=\sum_j (X_j')^2+\sum_l \gamma_l\phi(Y_l)=D_{\phi_*(\gamma)},$$
where $\phi_*(\gamma)\in M_N(\C)^{m'}$ is determined from $\sum_l \phi_*(\gamma)_lY_l'=\sum_l \gamma_l\phi(Y_l)$. Therefore $\phi_*:M_N(\C)^m\to M_N(\C)^{m'}$ is well defined.

If $\phi$ is surjective, the integrated map $\phi:G\to G'$ is also surjective and induces an injective map $\hat{\phi}:\widehat{G'}\to \widehat{G}$. Therefore, the represented operator $\pi(D_\gamma)$ is injective in any $\pi=\hat{\phi}(\pi')\in \Phi$ if and only if $\pi'(D_{\phi_*(\gamma)})=\pi\circ \phi(D_\gamma)$ is injective for any $\pi'\in \widehat{G'}\setminus \{1\}$. Therefore, $\phi_*(\gamma)\in \mathcal{E}(\mathfrak{g}',N)$ for $\gamma\in \mathcal{E}_\Phi(\mathfrak{g},N)$. The defining equation $\sum_l \phi_*(\gamma)_lY_l'=\sum_l \gamma_l\phi(Y_l)$ and surjectivity of $\phi$ implies that $\phi_*:\mathcal{E}_\Phi(\mathfrak{g},N)\to \mathcal{E}(\mathfrak{g}',N)$ is surjective.
\end{proof}

\begin{remark}
It follows from the proof of Proposition \ref{lknjnada} that when $\Phi=\hat{\phi} (\widehat{G'})\setminus \{1\}$ then for $\gamma\in M_N(\C)^m$, we have that $\gamma\in \mathcal{E}_\Phi(\mathfrak{g},N)$ if and only if $\phi_*(\gamma)\in \mathcal{E}(\mathfrak{g}',N)$. Or in other words, for $\Phi=\hat{\phi} (\widehat{G'})\setminus \{1\}$ we have an equality
\begin{equation}
\label{diendiend}
\mathcal{E}_\Phi(\mathfrak{g},N)=\{\gamma\in M_N(\C)^m: \phi_*(\gamma)\in \mathcal{E}(\mathfrak{g}',N)\}.
\end{equation}
\end{remark}

\begin{cor}
Let $\phi_j:\mathfrak{g}\to \mathfrak{g}(j)$ be a collection of graded, surjective Lie algebra homomorphisms onto Lie algebras such that $(\mathfrak{g}(j))_{-1}$ generate $\mathfrak{g}(j)$. Assume that
$$\widehat{G}=\cup_j \widehat{\phi_j}(\widehat{G_j}).$$
Then, equipping $(\mathfrak{g}(j))_{-1}$ with the inner product induced from $\phi_j$, we have the equality
$$\mathcal{E}(\mathfrak{g},N)=\cap_j \{\gamma\in M_N(\C)^m: (\phi_j)_*\gamma\in \mathcal{E}(\mathfrak{g}(j),N)\}.$$
\end{cor}

\begin{proof}
By assumption, $\mathcal{E}(\mathfrak{g},N)=\cap_j\mathcal{E}_{\Phi_j}(\mathfrak{g},N)$ where $\Phi_j=\widehat{\phi_j}(\widehat{G_j})\setminus \{1\}$. The corollary is now immediate from the observation in Equation \eqref{diendiend}.
\end{proof}

\subsection{Ramifications for the index of sub-Laplacians}

\begin{prop}
\label{lnljnjnad}
Assume that $\mathfrak{X}$ is a Carnot manifold of nilpotency degree $>2$ coming from an equiregular differential system. Then any sub-Laplacian $D_\gamma\in \mathcal{DO}_H^2(\mathfrak{X};E)$ (as in Equation \eqref{jnkjnkjnad}) on a vector bundle $E$ of rank $N$ is $H$-elliptic if and only if 
$$\gamma:(T^{-1}\mathfrak{X})^\perp/(T^{-2}\mathfrak{X})^\perp\to \End(E),$$
for any $x\in \mathfrak{X}$ satisfies that $\gamma_x, -\gamma_x^*\in \mathcal{E}(\mathfrak{t}_H\mathfrak{X}_x,N)$. In particular, if for any $x\in \mathfrak{X}$, $\mathfrak{t}_H\mathfrak{X}_x$ has property $\mathcal{E}^*N$ then whenever $D_\gamma$ is $H$-elliptic then $D_{t\gamma}$ is $H$-elliptic for any $t\in [0,1]$. 
\end{prop}

We note that by \cite{rothcrit}, the Rockland condition is equivalent to maximal hypoellipticity. 

\begin{proof}
It follows from the equivalence of the Rockland condition for $D_\gamma$ and $D_\gamma^*$ with H-ellipticity (see \cite{davehaller}) that $D_\gamma\in \mathcal{DO}_H^2(\mathfrak{X};E)$ is $H$-elliptic if and only if $\gamma_x,-\gamma_x^*\in \mathcal{E}(\mathfrak{t}_H\mathfrak{X}_x,N)$ for any $x\in \mathfrak{X}$. By definition, if $\mathfrak{g}$ has property $\mathcal{E}^*N$ then $t\mathcal{E}(\mathfrak{g},N)\subseteq\mathcal{E}(\mathfrak{g},N)$ for any $t\in (0,1]$ and by Proposition \ref{somsplpklad}, item (3), $0=0.\mathcal{E}(\mathfrak{g},N)\subseteq\mathcal{E}(\mathfrak{g},N)$. In particular, $t\mathcal{E}(\mathfrak{g},N)\subseteq\mathcal{E}(\mathfrak{g},N)$ for any $t\in [0,1]$ if $\mathfrak{g}$ has property $\mathcal{E}^*N$. The second conclusion follows.
\end{proof}

\begin{thm}
\label{indexzerofromegn}
Assume that $\mathfrak{X}$ is a compact Carnot manifold of nilpotency degree $>2$ coming from an equiregular differential system. If for any $x\in \mathfrak{X}$, $\mathfrak{t}_H\mathfrak{X}_x$ has property $\mathcal{E}^*N$, then any sub-Laplacian $D_\gamma\in \mathcal{DO}_H^2(\mathfrak{X};E)$ (as in Equation \eqref{jnkjnkjnad}) on a vector bundle $E$ of rank $N$ which is $H$-elliptic will have vanishing index. In particular, Problem \ref{mainprob} has a negative solution and we even have $[D_\gamma]=0$ in $K_0(\mathfrak{X})$.
\end{thm}

\begin{proof}
If $D_\gamma$ is $H$-elliptic, Proposition \ref{lnljnjnad} implies that 
the operator $D_{t\gamma}$ is $H$-elliptic for any $t\in [0,1]$. The path $(D_{t\gamma})_{t\in [0,1]}$ is continuous so 
$$\mathrm{index}(D_\gamma)=\mathrm{index}(D_0)=0,$$
since $D_0$ is a relatively compact perturbation of a self-adjoint operator. The argument that $[D_\gamma]=0$ in $K_0(\mathfrak{X})$ goes mutatis mutandis.
\end{proof}

\section{Sufficient conditions and necessary conditions} 
\label{seconestan}

\subsection{Results of Helffer and Rothschild-Stein}

Both Helffer \cite{helffercomp}, see also \cite{helffpdenil}, and Rothschild-Stein \cite{rothstein} provide sufficient conditions and necessary conditions for $\gamma\in \mathcal{E}(\mathfrak{g},1)$. Recall the construction of the inclusion $\iota:\mathfrak{g}_{-2}\to \mathfrak{so}(\mathfrak{g}_{-1})$ from Remark \ref{inclusionrem}. We present these results in a condensed form:

\begin{thm}
\label{lnonad}
Let $\mathfrak{g}$ be a graded nilpotent Lie algebra of nilpotency degree $r$ and generated by $\mathfrak{g}_{-1}$. Equip $\C^m$ with the norm induced from $\C^m\cong \mathfrak{g}_{-2}\otimes_\R\C\xrightarrow{\iota}\mathfrak{so}(\mathfrak{g}_{-1}\otimes_\R\C )$ and the operator norm on $\mathfrak{so}$. Identify $\R^m\subseteq \C^m$ the closed subspace closed under entrywise conjugation. 
\begin{enumerate}
\item If $\mathrm{dim}(\mathfrak{g}_{-2})\neq 2$ it holds that
$$\mathcal{E}(\mathfrak{g},1)=\{\gamma\in \C^m: \|\Im(\gamma)\|_{\R^m/\R\Re(\gamma))}<1\}.$$
\item Assume that $r>2$. We have that $\gamma\notin \mathcal{E}(\mathfrak{g},1)$ if $\Re(\gamma)=0$ and $\|\Im(\gamma)\|_{\C^m}\geq 1$.
\end{enumerate}
\end{thm}

Part 1) of this theorem is found as \cite[Theorem 1 and 2]{rothstein} and Part 2) of this theorem is stated as \cite[Theoreme 5]{helffercomp}. In \cite{helffercomp,helffpdenil,rothstein}, the results are stated in terms of the commutator description \eqref{blabla} of the operators $D_\gamma$. Said results translate to the setting of Theorem \ref{lnonad} using the mapping $\delta$ of Proposition \ref{knlknada}. We note that even though the statement of \cite[Theorem 2]{rothstein} is correct, its proof contains a minor gap that is filled in Section \ref{fillingap} below. In \cite[Theorem 2]{rothstein}, the necessity of $\|\Im(\gamma)\|_{\R^m/\R\Re(\gamma))}<1$ is proved under the condition that $(\mathfrak{g}_{-1}\oplus\mathfrak{g}_{-2})/(\R\sum_l \Re(\gamma)_lY_l)$ is not a Heisenberg algebra, which holds if $\mathrm{dim}(\mathfrak{g}_{-2})\neq 2$.

\begin{cor}
\label{cortors}
Any graded nilpotent Lie algebra $\mathfrak{g}$ such that $\mathfrak{g}_{-1}$ generates $\mathfrak{g}$ and $\mathrm{dim}(\mathfrak{g}_{-2})\neq 2$ has property $\mathcal{E}^*1$. 
\end{cor}

\begin{proof}
We need to verify that for any $t\in (0,1]$, we have an inclusion
$$t\mathcal{E}(\mathfrak{g},1)\subseteq\mathcal{E}(\mathfrak{g},1).$$
Using that $\R\Re(\gamma)=\R\Re(t\gamma)$ for any $t\in (0,1]$, Theorem \ref{lnonad} implies that 
\begin{align*}
t\mathcal{E}(\mathfrak{g},1)=&\{t\gamma\in \C^m: \|\Im(\gamma)\|_{\R^m/\R\Re(\gamma))}<1\}\\
=&\{t\gamma\in \C^m: \|\Im(t\gamma)\|_{\R^m/\R\Re(\gamma))}<t\}\\
\subseteq& \{t\gamma\in \C^m: \|\Im(t\gamma)\|_{\R^m/\R\Re(t\gamma))}<1\}\\
\subseteq& \{\gamma\in \C^m: \|\Im(\gamma)\|_{\R^m/\R\Re(\gamma))}<1\}=\mathcal{E}(\mathfrak{g},1).
\end{align*}

\end{proof}

For a real number $p>0$, write $H(p):=(-\infty,-p]\cup[p,\infty)\subseteq \C$.

\begin{prop}
\label{alknalkdnad}
Assume that $\mathfrak{g}$ is nilpotent of step length $>2$ and that $\mathfrak{g}_{-2}\equiv [\mathfrak{g}_{-1},\mathfrak{g}_{-1}]$ has dimension $1$. Then it holds that 
$$\mathcal{E}(\mathfrak{g},N)=\{\gamma\in M_N(\C): H(1)\cap \mathrm{Spec}(i\gamma)=\emptyset\}.$$
In particular, any nilpotent Lie algebra $\mathfrak{g}$ of step length $>2$ and with $\mathfrak{g}_{-2}$ one-dimensional has property $\mathcal{E}^*N$ for any $N$.
\end{prop}

\begin{proof}
First, we consider $N=1$. Since $\mathfrak{g}_{-2}$ has dimension $1$, $\R=\R\Re(\gamma)$ as soon as $\Re(\gamma)\neq 0$ and in this case $\gamma\in \mathcal{E}(\mathfrak{g},1)$ by part 1 of Theorem \ref{lnonad}. If $\Re(\gamma)= 0$, Theorem \ref{lnonad} implies that $\gamma\in \mathcal{E}(\mathfrak{g},1)$ if and only if $\|\Im(\gamma)\|_{\R}<1$. We conclude that $\mathcal{E}(\mathfrak{g},1)=\{\gamma\in \C: i\gamma\notin H(1)\}$ and the corollary follows for $N=1$.

For $N>1$, we can up to an invertible matrix in $\C^N$, assume that $\gamma\in M_N(\C)$ is in Jordan form. Without restriction, we can assume that $\gamma$ is in fact one Jordan block
$$\gamma=
\begin{pmatrix} 
\lambda& 1&0&0&\cdots\\
0&\lambda& 1&0&\cdots\\
\vdots&0&\ddots &\ddots &\vdots\\
0&&&\lambda &1\\
0&&&&\lambda
\end{pmatrix},$$
for some $\lambda\in \C$. Considering the Rockland condition on the bottom right corner, we see from the case $N=1$ that if $\gamma\in \mathcal{E}(\mathfrak{g},N)$ then $i\lambda\notin H(1)$. Conversely, a linear algebra argument and the Rockland condition shows that if $i\lambda\notin H(1)$ then $\gamma\in \mathcal{E}(\mathfrak{g},N)$. The corollary follows.
\end{proof}

\subsection{Related results using a quotient map}

We say that a two-step nilpotent Lie group $\mathfrak{p}=\mathfrak{p}_{-1}\oplus \mathfrak{p}_{-2}=\mathfrak{p}_{-1}\oplus \mathfrak{z}$ is polycontact if for any $\xi\in \mathfrak{z}^*\setminus \{0\}$ the Kirillov form 
$$\omega_\xi(X,Y):=\xi[X,Y],$$
is non-degenerate on $\mathfrak{p}_{-1}=\mathfrak{g}/\mathfrak{z}$.  We write $\R[-1]$ for the graded abelian Lie algebra $\R$ concentrated in degree $-1$. For an inner product space $V$, we write $S(V)$ for its sphere.

\begin{prop}
\label{gener}
Assume that $\mathfrak{g}$ is nilpotent and there is a graded Lie algebra quotient $\phi:\mathfrak{g}\to \mathfrak{p}\oplus \R[-1]$ where $\mathfrak{p}$ is a step $2$ polycontact Lie algebra. Let $m:=\dim(\mathfrak{g}_{-2})$. We set $p:=\dim(\mathfrak{p}_{-1})/2$, $m_\mathfrak{p}:=\dim(\mathfrak{p}_{-2})$ and write $\alpha$ for the smallest positive singular value of the quotient map of inner product spaces $\mathfrak{g}_{-2}\to \mathfrak{p}_{-2}$. Then the constant $ p_\mathfrak{g}:=p\alpha$ satisfies that 
$$\mathcal{E}(\mathfrak{g},N)\subseteq \{\gamma\in M_N(\C)^m: H(p_{\mathfrak{g}}^0)\cap \mathrm{Spec}(\gamma(\xi))=\emptyset\quad \forall \xi\in S(\mathrm{im}(\phi^*)\cap \mathfrak{g}_{-2}^*)\}.$$
In fact,
$$\mathcal{E}(\mathfrak{p}\oplus \R[-1],N)= \{\gamma\in M_N(\C)^{m_{\mathfrak{p}}}: H(p)\cap \mathrm{Spec}(\gamma(\xi))=\emptyset\quad \forall \xi\in S(\mathfrak{p}_{-2}^*)\}.$$
\end{prop}

\begin{proof}
By Proposition \ref{lknjnada}, it suffices to compute $\mathcal{E}(\mathfrak{p}\oplus \R[-1],N)$. Indeed, linearity of $\xi\mapsto \gamma(\xi)$ and scaling properties of the set-valued function 
$$M_N(\C)^m\ni \gamma\mapsto \cup_{\xi\in S(\mathrm{im}(\phi^*)\cap \mathfrak{g}_{-2}^*)}\mathrm{Spec}(\gamma(\xi)),$$
imply that the constant $p_\mathfrak{g}:=p\alpha$ has the sought after property. Write $X_{0}$ for the relevant basis element of $\R[-1]$ so
$$D_\gamma=X_0^2+\sum_{j=1}^{2p} X_j^2+\sum_{l=1}^{m}\gamma_lY_l.$$

 Let $P$ denote the simply connected Lie group integrating $\mathfrak{p}$. The unitary, irreducible representations of $P\times \R$ are parametrized by $\R^*\times \mathfrak{p}^*/\mathrm{Ad}(P)$. Since $P$ is polycontact, we have that $\mathfrak{p}^*/\mathrm{Ad}(P)=\mathfrak{p}_{-1}^*\sqcup \mathfrak{p}_{-2}^*\setminus \{0\}$, where $\mathfrak{p}_{-1}^*$ induces the characters and $\mathfrak{p}_{-2}^*\setminus \{0\}$ induces the flat orbit representations. For a non-trivial character $\pi=(\xi_0,\xi_1)\in \R^*\times \mathfrak{p}_{-1}^*$, $\pi(D_\gamma)=\xi_0^2+|\xi_1|^2>0$ and it remains to characterize the injectivity of $\pi(D_\gamma)$ for $\pi=(\xi_0,\xi_2)\in \R^*\times \mathfrak{p}_{-2}^*\setminus \{0\}$. A short computation shows that for such representations, 
$$\pi(D_\gamma)=\xi_0^2+|\xi_2|H+\sum_{l=1}^{m}i\xi_{2,l}\gamma_l=\xi_0^2+|\xi_2|H+\gamma(\xi_2),$$
where $\xi_2=(\xi_{2,1},\ldots,\xi_{2,m})$ in the basis $Y_1,\ldots, Y_m$ and $H$ denotes the harmonic oscillator on $\R^p$. As above, we write $\gamma(\xi)=i\sum_{l=1}^{m}\xi_{2,l}\gamma_l$. Using the auxiliary variable $t=\xi_0^2/|\xi_2|$ we conclude from the spectrum of the harmonic oscillator that $\pi(D_\gamma)$ is injective if and only if $\mathrm{Spec}(\gamma(\xi_2/|\xi_2|))$ does not intersect the set 
$$\cup_{t\geq 0, k\in \N} \{(t+k+p),-(t+k+p)\}=H(p).$$
Therefore, $\pi(D_\gamma)$ is injective for all $\pi \in \R^*\times \mathfrak{p}_{-2}^*\setminus \{0\}$ if and only if $\gamma=(\gamma_1,\ldots,\gamma_l)$ satisfies 
$$\mathrm{Spec}(\gamma(\xi))\cap H(p)=\emptyset,$$
for all $\xi\in \mathfrak{p}_{-2}^*\setminus \{0\}$ of unit length.
\end{proof}

\begin{cor}
\label{conjecjnjnpartial}
Let $\mathfrak{g}$ be a graded nilpotent Lie algebra such that $\mathfrak{g}_{-1}$ generates $\mathfrak{g}$. Assume that there are quotient maps $\phi(j):\mathfrak{g}\to \mathfrak{p}(j)\oplus \R[-1]$ for polycontact step two nilpotent Lie algebras $\mathfrak{p}(j)$ such that 
$$S(\mathfrak{g}_{-2}^*)=\cup_j S(\mathrm{im}(\phi(j)^*)\cap \mathfrak{g}_{-2}^*).$$
Then there is a real number $ p_\mathfrak{g}> 0$ such that 
$$\mathcal{E}(\mathfrak{g},N)\subseteq \{\gamma\in M_N(\C)^m: H(p_{\mathfrak{g}}^0)\cap \mathrm{Spec}(\gamma(\xi))=\emptyset\quad \forall \xi\in S(\mathfrak{g}_{-2}^*)\}.$$
\end{cor}

\subsection{Paraphrasing Rotschild-Stein's theorem when $N>1$}
\label{extrs}

In this subsection we shall partially extend Rotschild-Stein's theorem (summarized in Theorem \ref{lnonad} above) to the matrix case.
Equip $M_N(\C)^m$ with the norm induced from 
$$M_N(\C)^m\cong \mathfrak{g}_{-2}\otimes_\R M_N(\C)\xrightarrow{\iota}\mathfrak{so}(\mathfrak{g}_{-1}\otimes_\R M_N(\C) ),$$ 
and the operator norm on $\mathfrak{so}$. Write $W$ for the subspace of $M_N(\C)^m$ of elements invariant under entrywise hermitean conjugate. 

\begin{thm}
\label{lnonadhighN}
Let $\mathfrak{g}$ be a graded nilpotent Lie algebra generated by $\mathfrak{g}_{-1}$. It holds that
$$\mathcal{E}(\mathfrak{g},N)\supseteq \{\gamma\in M_N(\C)^m: \|\Im(\gamma)\|_{W/\R\Re(\gamma))}<1\}.$$
\end{thm}

Our argument follows that of \cite[Theorem 1]{rothstein} closely.

\begin{proof}
Using the Rockland condition (or standard techniques) it suffices to show that if $ \|\Im(\gamma)\|_{W/\R\Re(\gamma))}<1$, there is a constant $A>0$ such that
\begin{equation}
\label{lknlknlknad}
\sum_{j=1}^n \|X_j f\|^2_{L^2(G,\C^N)}\leq A|\langle f,D_\gamma f\rangle|, \quad \forall f\in C^\infty_c(G,\C^N).
\end{equation}

For notational simplicity we write $\gamma=a+ib$ where $a\equiv \Re(\gamma),b\equiv \Im(\gamma)\in M_N(\C)^m$ are self-adjoint. By the same argument as in \cite[Lemma 2.7]{rothstein}, there is a $t\in \R$ such that $\|b+ta\|_{W}=\|b\|_{W/\R a}$.

We can write 
\begin{align*}
\langle f, D_\gamma f\rangle=&\sum_{j=1}^n \|X_j f\|^2_{L^2(G,\C^N)}+\sum_{l=1}^m \left[\langle f, a_lY_lf\rangle+i\langle f, b_lY_lf\rangle\right]=\\
=&\sum_{j=1}^n \|X_j f\|^2_{L^2(G,\C^N)}+(1-it)\sum_{l=1}^m\langle f, a_lY_lf\rangle+i\sum_{l=1}^m\langle f, (b_l+ta_l)Y_lf\rangle
\end{align*}
Rearranging these terms and using the triangle equality give the estimate
\begin{equation}
\label{tolest}
\sum_{j=1}^n \|X_j f\|^2_{L^2(G,\C^N)}\leq |\langle f,D_\gamma f\rangle|+|1-it|\left|\sum_{l=1}^m\langle f, a_lY_lf\rangle\right|+\left|\sum_{l=1}^m\langle f, (b_l+ta_l)Y_lf\rangle\right|.
\end{equation}

To estimate the right hand side, we make the following observations. We have that 
$$\Im(\langle f, D_\gamma f\rangle)=\sum_{l=1}^m\langle f, a_lY_lf\rangle,$$
and therefore
\begin{equation}
\label{realest}
\left|\sum_{l=1}^m\langle f, a_lY_lf\rangle\right|\leq |\langle f,D_\gamma f\rangle|.
\end{equation}
Moreover, using Proposition \ref{knlknada} we write 
$$\sum_{l=1}^m\langle f, (b_l+ta_l)Y_lf\rangle=\sum_{k,l=1}^{n}\langle f,\delta(b+ta)_{k,l} [X_k,X_l]f\rangle=\mathrm{Tr}(\delta(b+ta)\rho),$$
where 
$$\rho=\int_G ([X_k,X_l]f)(g)\otimes f(g)^*\mathrm{d} g\in \mathfrak{so}_n(M_N(\C)).$$
Following the same argument as in \cite[Lemma 2.7]{rothstein}, the trace norm of $\rho$ satisfies the estimate 
$$\|\rho\|_{\mathcal{L}^1}\leq \sum_{j=1}^n \|X_j f\|^2_{L^2(G,\C^N)}.$$
We therefore have the upper bound
\begin{align}
\label{imlest}
\left|\sum_{l=1}^m\langle f, (b_l+ta_l)Y_lf\rangle\right|=&|\mathrm{Tr}(\delta(b+ta)\rho)|\leq \\
\nonumber
\leq &\|\delta(b+ta)\|_{M_n(M_N(\C))}\sum_{j=1}^n \|X_j f\|^2_{L^2(G,\C^N)}=\|b\|_{W/\R a}\sum_{j=1}^n \|X_j f\|^2_{L^2(G,\C^N)}.
\end{align}
If we rearrange the estimate \eqref{tolest} using the estimates \eqref{realest} and \eqref{imlest} we arrive at the desired estimate \eqref{lknlknlknad} for the constant 
$$A=\frac{1+|1+it|}{1-\|\Im(\gamma)\|_{W/\R\Re(\gamma))}}.$$
\end{proof}

We now turn to partial converses of Theorem \ref{lnonadhighN}. Given a graded nilpotent Lie algebra $\mathfrak{g}=\oplus_{j=1}^r\mathfrak{g}_{-j}$ the truncated Lie algebra $\mathfrak{g}_{-1}\oplus \mathfrak{g}_{-2}$ of nilpotency degree $2$ is defined by declaring $\mathfrak{g}_{-2}$ central and the Lie bracket $\mathfrak{g}_{-1}\wedge \mathfrak{g}_{-1}\to \mathfrak{g}_{-2}$ defined from that on $\mathfrak{g}$. We note that if $\mathfrak{g}_{-1}$ generates $\mathfrak{g}$, then $\mathfrak{g}_{-1}$ generates the truncated Lie algebra $\mathfrak{g}_{-1}\oplus \mathfrak{g}_{-2}$ and the center of $\mathfrak{g}_{-1}\oplus \mathfrak{g}_{-2}$ is $\mathfrak{g}_{-2}$. In particular, a flat coadjoint orbit for $\mathfrak{g}_{-1}\oplus \mathfrak{g}_{-2}$ corresponds to a coadjoint equivalence class of $\xi\in \mathfrak{g}_{-2}^*$ such that the Kirillov form $\omega_\xi(X,Y):=\xi[X,Y]$ is non-degenerate on $\mathfrak{g}_{-1}$.

\begin{thm}
\label{spechighN}
Let $\mathfrak{g}$ be a graded nilpotent Lie algebra generated by $\mathfrak{g}_{-1}$. If the Lie algebra $\mathfrak{g}_{-1}\oplus \mathfrak{g}_{-2}$ truncated to nilpotency degree 2 does not admit any flat coadjoint orbits, then
$$\mathcal{E}(\mathfrak{g},N)\subseteq \{\gamma\in M_N(\C)^m: \mathrm{Spec}(\gamma(\xi))\cap H(1)=\emptyset \quad \forall \xi\in S(\mathfrak{g}_{-2}^*)\}.$$
\end{thm}

\begin{proof}
We prove that if $\gamma\in M_N(\C)^m$ satisfies that $\gamma(\xi)$ has an eigenvalue $\leq -1$ for some $\xi\in S(\mathfrak{g}_{-2}^*)$ then $D_\gamma$ is not hypoelliptic. Indeed, if this is the case, we take such a $\xi$ and the eigenvector $v$ of $\gamma(\xi)$ with eigenvalue $\leq -1$. Since $\mathfrak{g}_{-1}\oplus \mathfrak{g}_{-2}$ admits no flat coadjoint orbits, $\xi\in \mathfrak{g}^*_{-2}$ defines a degenerate Kirillov form on $\mathfrak{g}_{-1}$ and the argument in \cite[Proof of Theorem 2, Section 4]{rothstein} produces a non-trivial irreducible unitary representation $(\pi,\mathcal{H}_\pi)$ of $G$ and a vector $H\in \mathcal{H}_\pi$ such that $\pi(D_{\gamma})(H\otimes v)=0$. In particular, $D_{\gamma}$ is not hypoelliptic. 
\end{proof}

\begin{definition}
Let $\gamma\in \mathfrak{g}_{-2}\otimes M_N(\C)$ be written as $\gamma=a+ib$ for $a$ and $b$ self-adjoint. We say that $\gamma$ has {\bf property $\alpha$} if for any  unit vectors $\xi\in \mathfrak{g}_{-2}^*$ and $v\in \ker(\xi(a))$ we have that $|\langle v, \xi(b)v\rangle_{\C^N}|<1$. 
\end{definition}

\begin{thm}
\label{chadhighN}
Let $\mathfrak{g}$ be a graded nilpotent Lie algebra generated by $\mathfrak{g}_{-1}$. If the Lie algebra $\mathfrak{g}_{-1}\oplus \mathfrak{g}_{-2}$ truncated to nilpotency degree 2 does not admit any flat coadjoint orbits, then
$$\mathcal{E}(\mathfrak{g},N)\subseteq \{\gamma\in M_N(\C)^m:\; \gamma\; \mbox{has property $\alpha$}\}.$$
\end{thm}

By Theorem \ref{lnonadhighN} it suffices to show that if $\gamma$ does not have property $\alpha$, then $D_\gamma$ is not hypoelliptic. In fact, this can be seen immediately from Theorem \ref{spechighN} since if $\gamma$ does not have property $\alpha$ then $\mathrm{Spec}(\gamma(\xi))\cap H(1)\neq\emptyset$ for some unit vector $\xi\in \mathfrak{g}_{-2}^*$. We will however take the chance to expand the argument and in parallel showcase the difference to the converse inclusion of Theorem \ref{lnonadhighN}. Our argument follows that of \cite[Theorem 2]{rothstein} closely.

For notational simplicity we write $\gamma=a+ib$ where $a\equiv \Re(\gamma),b\equiv \Im(\gamma)\in M_N(\C)^m$ are self-adjoint. We note that $\|b\|_{W/\R a}$ is the maximal value of $\rho\mapsto |\mathrm{Tr}(\rho b)|$ where $\rho\in M_N(\C)^m$ ranges over the trace norm sphere $\|\rho\|_{\mathcal{L}^1}=1$ and satisfying $\mathrm{Tr}(\rho a)=0$. So by convexity of $\rho\mapsto |\mathrm{Tr}(\rho b)|$ the maximum is attained in an extremal point $\xi\otimes e$ for $\xi\in \mathfrak{g}_{-2}^*$ of trace norm $1$ and $e\in M_N(\C)^*$ a vector state. We shall assume that $\mathrm{Tr}(\rho b)=-\|b\|_{W/\R a}\leq -1$. Since $e$ is a vector state, the fact that $\mathrm{Tr}_{M_N(\C)^m}(\rho b)=\mathrm{Tr}_{M_N}(e \xi(b))\leq -1$ and the min-max principle implies that the self-adjoint $\xi(b)\in M_N(\C)$ has an eigenvector $v$ with eigenvalue $\leq -1$.

Since $\mathfrak{g}_{-1}\oplus \mathfrak{g}_{-2}$ admits no flat coadjoint orbits, $\xi\in \mathfrak{g}^*_{-2}$ defines a degenerate Kirillov form on $\mathfrak{g}_{-1}$ and the argument in \cite[Proof of Theorem 2, Section 4]{rothstein} produces an irreducible unitary representation $(\pi,\mathcal{H}_\pi)$ of $G$ and a vector $H\in \mathcal{H}_\pi$ such that $\pi(D_{ib})(H\otimes v)=0$. In particular, $D_{ib}$ is not hypoelliptic. If we in addition assume that $\gamma$ does not have property $\alpha$, we can assume that the vector state $e\in M_N(\C)$ is of the form $e=v_0\otimes v_0^*$ where $\xi(a)v_0=0$. As such, the argument used in the proof of Theorem \ref{spechighN} coming from \cite[Theorem 2]{rothstein} carries over and disproves hypoellipticity of $D_\gamma$.

\begin{question}
In the statement of Theorem \ref{lnonadhighN}, can the property $\|\Im(\gamma)\|_{W/\R\Re(\gamma))}<1$ be determined from pointwise spectral properties of the first degree polynomial $\gamma(\xi)$ on $\mathfrak{g}_{-2}^*$ similar to the condition in Theorem \ref{spechighN}?
\end{question}

\section{Examples}
\label{ljnljnad}

To describe the irreducible, unitary representations we use the Kirillov orbit method: $\widehat{G}=\mathfrak{g}^*/\mathsf{Ad}^*$. In all examples we focus on describing the situation in a set of generic orbits $\Gamma\subseteq \widehat{G}$ that will be an open, dense, Hausdorff subset. In all cases, hypoellipticy can be characterized from a ``good enough'' invertibility argument in the generic orbits.

We first recall the notion of flat orbits. Writing $\mathfrak{z}$ for the center of $\mathfrak{g}$, a $\xi\in \mathfrak{g}^*$ belongs to a flat coadjoint orbit if the induced Kirillov form $\omega_\xi(X,Y):=\xi([X,Y])$ is non-degenerate on $\mathfrak{g}/\mathfrak{z}$. If a flat orbit exists, we take the set of generic orbits $\Gamma\subseteq \widehat{G}$ as the set corresponding to all flat orbits; if it is non-empty it is a Zariski-open, dense, Hausdorff subset that can be identified with the open subset
$$\Gamma=\{\xi\in \mathfrak{z}^*: \mathrm{Pf}(\omega_\xi)\neq 0\}.$$
Here $\mathrm{Pf}$ denotes the Pfaffian, so the polynomial $\xi\mapsto \mathrm{Pf}(\omega_\xi)$ takes a non-zero value at $\xi$ if and only if $\omega_\xi$ is non-degenerate.

\subsection{Step 2 with flat orbits}

Consider the case when $\mathfrak{g}$ is step $2$, and admits a flat orbit. Consider the trivial vector bundle $V:=\Gamma\times \mathfrak{g}/\mathfrak{z}$ over $\Gamma$ and chose a metric thereon. The Kirillov form induces a symplectic form on $V$. The metric and the Kirillov form induces a canonical complex structure on $V$,  which is adapted to a potentially different metric. We form the bosonic Fock space bundle 
$$\mathfrak{F}_V:=\bigoplus_{k=0}^\infty V^{\otimes^{\rm sym}_\C k}.$$
This is a Hilbert space bundle over $\Gamma$. Its relevance comes from the fact that the complex structure and the Kirillov orbit method gives $\mathfrak{F}_V$ the structure of a bundle of representations over $\Gamma$. Using the metric, we define the complex linear symmetric operator $g_\omega=|\omega|=\sqrt{-\omega^2}$. We define the endomorphisms 
$$
s_k:V^{\otimes_\C^{\rm sym}k}\to V^{\otimes_\C^{\rm sym}k},
$$
on the $k$:th symmetric tensor power, from 
$$s_k(v_1\otimes_\C^{\rm sym}\cdots \otimes_\C^{\rm sym}v_k):=\sum_{j=1}^kv_1\otimes_\C^{\rm sym}\cdots\otimes_\C^{\rm sym}g_\omega v_j \otimes_\C^{\rm sym}\cdots\otimes_\C^{\rm sym}v_k.$$
 If we use a metric on $V$ making the complex structure defined from the Kirillov form adapted, then $g_\omega$  is the identity operator and $s_k$ is $k$ times the identity operator. We assume that $\mathfrak{g}_{-2}=\mathfrak{z}$ in which case we can identify $\mathfrak{z}$ with a subspace of $\mathfrak{so}(\mathfrak{g}_{-1})$. Recall the following theorem from \cite{goffkuz}.

\begin{thm}
Let $\mathfrak{g}$ be a step $2$ nilpotent Lie group with flat orbits and $\mathfrak{g}_{-2}=\mathfrak{z}$ of dimension $m$. Fix a metric on $\mathfrak{g}_{-1}$ and an orthonormal basis $X_1,\ldots, X_n$. Then $\mathcal{E}(\mathfrak{g},N)$ consists of those $\gamma\in M_N(\C)^m$ such that 
$$\gamma_k:\Gamma\ni \xi\mapsto s_k(\xi)+\frac{\mathrm{Tr}(|\omega_\xi|)}{2}+\gamma(\xi)\in \End(V^{\otimes_\C^{\rm sym}k}\otimes \C^N),
$$
satisfies that $\gamma_k(\xi)$ is invertible for all $\xi\in \Gamma$ and 
$$\sup_{k,\xi\in \Gamma} \left\|\left(s_k(\xi)+\frac{\mathrm{Tr}(|\omega_\xi|)}{2}\right)\left(s_k(\xi)+\frac{\mathrm{Tr}(|\omega_\xi|)}{2}+\gamma(\xi)\right)^{-1}\right\|_{\End(V^{\otimes_\C^{\rm sym}k}\otimes \C^N)}<\infty.$$
\end{thm}

\begin{remark}
In the special case that $\mathfrak{g}$ is polycontact, i.e. $\Gamma=\mathfrak{z}^*\setminus \{0\}$, the homogeneity and compactness of the sphere in $\mathfrak{z}^*$ ensure that $\mathcal{E}(\mathfrak{g},N)$ consists of those $\gamma\in \mathfrak{so}_n(M_N(\C))$ such that $\gamma_k$  is invertible for all $\xi\in \Gamma$.

In general, $\Gamma$ is noncompact (e.g. $\Gamma=(\R^\times)^2$ for a product of two Heisenberg groups). Therefore the bound on the inverse is in general not superfluous. 
\end{remark}

\subsection{The Engel-Lie algebra}

Consider the four dimensional Lie algebra $\mathfrak{g}$ spanned by $X_1,X_2,X_3,X_4$ where the non-zero brackets are given by 
$$[X_1,X_2]=X_3 \quad\mbox{and}\quad [X_1,X_3]=X_4.$$
This Lie algebra is called the Engel-Lie algebra as it relates to Engel structures on $4$-manifolds (whose existence is equivalent to parallellizability). The Lie algebra $\mathfrak{g}$ has nilpotency degree $3$ with $\mathfrak{g}_{-1}=\R X_1+\R X_2$, $\mathfrak{g}_{-2}=\R X_3$, and $\mathfrak{g}_{-3}=\mathfrak{z}=\R X_4$. Proposition \ref{alknalkdnad} implies the following. Recall our notation $H(p)=(-\infty,-p]\cup [p,\infty)$.

\begin{prop}
The four dimensional, three step nilpotent Lie algebra $\mathfrak{g}$ satisfies that
$$\mathcal{E}(\mathfrak{g},N)=\{\gamma\in M_N(\C): H(1)\cap \mathrm{Spec}(i\gamma)=\emptyset\}.$$
\end{prop}

To get a better feeling, let us describe the situation in more detail. The representation space $\widehat{G}$ of the Engel-Lie group was described in detail in \cite[Example 1.3.10 and 2.2.2]{Corwin_Greenleaf} and \cite[Chapter 3.3]{kirillovbook}. By \cite[page 80]{kirillovbook}, 
$$\widehat{G}=\Gamma_1\dot{\cup} \Gamma_2\dot{\cup} \Gamma_3, \quad
\mbox{where}\quad 
\begin{cases}
\Gamma=\Gamma_1=\R_p^\times \times \R_q,\\
\Gamma_2=\{0_p^+,0_p^-\}\times \R_{q,<0},\\
\Gamma_3=\{0_{pq}\}\times \R^2,
\end{cases}$$
are all subsets that are Hausdorff in their respective induced topologies. Here $p$ and $q$ refer to representation parameters coming from the polynomial invariants $p=X_4$ and $q=2X_2X_4-X_3^2$ for the coadjoint action. 

We now describe $\pi(D_\gamma)$ where $\pi$ ranges over $\widehat{G}$. Here 
$$D_\gamma=X_1^2+X_2^2+\gamma X_3.$$
For the characters $(\xi_1,\xi_2)\in \Gamma_3$, 
$$\pi_{(\xi_1,\xi_2)}(D_\gamma)=-\xi_1^2-\xi_2^2,$$ 
which is invertible in all non-trivial characters. For $\pi_{0,q}^\pm \in \Gamma_2=\{0_p^+,0_p^-\}\times \R_{q,<0}$, the functional dimension is $1$ and the computation on \cite[page 83]{kirillovbook} shows that 
$$\pi_{0,q}^\pm(D_\gamma)=\partial_x^2+4\pi^2qx^2\pm 2\pi \sqrt{-q}\gamma=-2\pi \sqrt{-q}(H\pm \gamma),$$
where $H$ is unitarily equivalent to a harmonic oscillator with spectrum $2\N+1$. In the generic orbits $\pi_{p,q}\in \Gamma=\Gamma_1=\R_p^\times\times \R_{q}$, the functional dimension is $1$ and the computation on \cite[page 83]{kirillovbook} shows that  $\pi_{p,q}(D_\gamma)$ is the anharmonic oscillator
$$\pi_{p,q}(D_\gamma)=\partial_x^2-4\pi^2\left(px^2+\frac{q}{2p}\right)^2+2\pi p\gamma x.$$
There is a detailed analysis of invertibility of $\pi_{p,q}(D_\gamma)$ in \cite{helffercomp,helffpdenil} confirming the analysis above that invertibility is equivalent to $i\gamma\notin H(1)$.

\subsection{$N(4)$}
\label{n4ex}
Consider the six dimensional Lie algebra $\mathfrak{n}(4)$ spanned by $X_1,X_2,X_3,Y_1,Y_2,Z$ where the non-zero brackets are given by 
$$[X_1,X_2]=Y_1, [X_2,X_3]=Y_2 \quad\mbox{and}\quad [X_1,Y_2]=[Y_1,X_3]=Z.$$
This is the Lie algebra of the group $N(4)$ of all real upper unipotent $4\times 4$-matrices.  The group $N(4)$ is the unipotent radical in the minimal parabolic subgroup of $SL_4(\R)$. The Lie algebra $\mathfrak{n}(4)$ has nilpotency degree $3$ with \\ $\mathfrak{n}(4)_{-1}=\R X_1+\R X_2+\R X_3$, $\mathfrak{n}(4)_{-2}=\R Y_1+\R Y_2$, and $\mathfrak{n}(4)_{-3}=\mathfrak{z}=\R Z$.

\begin{prop}
\label{alakdnjak}
The six dimensional, three step nilpotent Lie algebra $\mathfrak{n}(4)$ satisfies that
$$\mathcal{E}(\mathfrak{n}(4),N)\subseteq \{\gamma\in M_N(\C)^2: H(1)\cap \mathrm{Spec}(\gamma(\xi))=\emptyset, \; |\xi|=1\},$$
and $\mathfrak{n}(4)$ has property $\mathcal{E}^*N$ for any $N$.
\end{prop}

In this example, the necessary and sufficient condition in Theorem \ref{lnonad} does not stand in dichotomy as it did when $\dim\mathfrak{g}_{-2}=1$ in Proposition \ref{alknalkdnad}. 

\begin{proof}
The representation space $\widehat{N(4)}$ was described in detail in \cite[Example 1.3.11 and 2.2.8]{Corwin_Greenleaf}. By \cite[Example 1.3.11]{Corwin_Greenleaf}, 
$$\widehat{N(4)}=\Gamma_1\,\dot{\cup}\, \Gamma_2\,\dot{\cup} \,\Gamma_3\,\dot{\cup}\, \Gamma_4,$$
where the space of generic orbits take the form $\Gamma_1=\R^\times \times \R$, and all form subsets that are Hausdorff in their respective induced topologies. The functional dimension of the representations in $\Gamma_1$ is $2$, in $\Gamma_2$ and $\Gamma_3$ it is $1$ and $\Gamma_4\cong \R^3$ consists of characters.

We want to characterize invertibility of $\pi(D_\gamma)$. Here 
$$D_\gamma=X_1^2+X_2^2+X_3^2+\gamma_1 Y_1+\gamma_2Y_2.$$ 
Trivially, $\pi(D_\gamma)$ is invertible for any non-trivial character $\pi\in \Gamma_4\setminus \{0\}$. By the computations in \cite[Example 2.2.8]{Corwin_Greenleaf}, all representations in $\Gamma_2\dot{\cup} \Gamma_3$ are induced from quotients onto $\mathfrak{h}_1\oplus \R[-1]$. Therefore Proposition \ref{gener} implies that $\pi(D_\gamma)$ is injective for all $\pi \in \Gamma_2\,\dot{\cup}\, \Gamma_3$ if and only if 
$$\mathrm{Spec}(\xi_1\gamma_1+\xi_2 \gamma_2)\cap \left((-i\infty,-i]\cup [i,i\infty)\right)=\emptyset,$$ 
for any $\xi$ such that $\xi_1^2+\xi_2^2=1$. \\  We see that 
$$\mathcal{E}_{ \Gamma_2\dot{\cup} \Gamma_3\dot{\cup} \Gamma_4\setminus \{1\}}(\mathfrak{n}(4),N)=\{\gamma\in M_N(\C)^2: H(1)\cap \mathrm{Spec}(\gamma(\xi))=\emptyset, \; |\xi|=1\}.$$

We now turn our focus to the generic orbits $\Gamma_1$. For an $(\alpha,\eta)\in \R^\times\times\R\cong \Gamma_1$, the functional dimension is $2$ and the computation \cite[Example 2.2.8]{Corwin_Greenleaf} gives us that 
$$\pi_{\alpha,\eta}(D_\gamma)=-\Delta-4\pi^2(\eta-\alpha xy)^2-2\pi i\alpha (\gamma_2x-\gamma_1 y).$$
This operator was studied in \cite{helffjorg}. We use the geometer's convention $\Delta=-\partial_x^2-\partial_y^2$. For notational simplicity, introduce the operator $H_{\alpha,\eta}(\gamma):=\pi_{\alpha,\eta}(D_\gamma)$. It is a self-adjoint operator with discrete spectrum on $L^2(\R^2)$ for $(\alpha,\eta)\in \R^\times\times\R$. In the new variables $u=t^{1/2}x$ and $v=t^{1/2}y$, 
\begin{align*}
H_{\alpha,\eta}(t\gamma)&=t^{-1}\left(-\Delta_{u,v}-4\pi^2(\eta t^{1/2}-\alpha t^{-1/2} uv)^2-2\pi i\alpha t^{-1/2} (\gamma_2u-\gamma_1 v\right)=\\
&=t^{-1}H_{t^{-1/2}\alpha,t^{1/2}\eta}(\gamma).
\end{align*}
Combining this identity, with the computation of $\mathcal{E}_{ \Gamma_2\dot{\cup} \Gamma_3\dot{\cup} \Gamma_4\setminus \{1\}}(\mathfrak{n}(4),N)$ above and the fact that $H_{\alpha,\eta}(0)$ is injective for any $(\alpha,\eta)\in \R^\times\times\R$ (by Hörmander's sum of squares theorem) we see that if $\gamma\in \mathcal{E}(\mathfrak{n}(4),N)$ then $t\gamma\in \mathcal{E}(\mathfrak{n}(4),N)$ for any $t\in [0,1]$. 
\end{proof}

\begin{question}
Can we characterize the subset $\mathcal{E}(\mathfrak{n}(4),N)\subseteq M_N(\C)^2$ more precisely? This subset consists precisely of the matrices $\gamma=(\gamma_1,\gamma_2)\in M_N(\C)^2$ such that additionally to the requirement that $H(1)\cap \mathrm{Spec}(\gamma(\xi))=\emptyset$, for $|\xi|=1$, the operator 
$$-\Delta-4\pi^2(\eta-\alpha xy)^2-2\pi i\alpha (\gamma_2x-\gamma_1 y) $$
is invertible for any $(\alpha,\eta)\in \R^\times\times\R$. 

Using the scaling invariance argument from the proof of Proposition \ref{alakdnjak} and a reflection in $(x,y)$ and $\eta$, it follows that $\gamma=(\gamma_1,\gamma_2)\in \mathcal{E}(\mathfrak{n}(4),N)$ if and only if $H(1)\cap \mathrm{Spec}(\gamma(\xi))=\emptyset$, for $|\xi|=1$, and the operator
\begin{align*}
&-\Delta-4\pi^2(\eta- xy)^2-2\pi i (\gamma_2x-\gamma_1 y),
\end{align*}
is invertible on $L^2(\R^2)$ for any $\eta\in\R$.
\end{question}

\subsection{$\tilde{\mathfrak{h}}_{2m+1}$}

Let us consider another step $3$ nilpotent Lie algebra. First, consider the Heisenberg Lie algebra $\mathfrak{h}_{2m+1}$ of dimension $2m+1$ generated by $X_1,\ldots, X_m, Y_1,\ldots, Y_m,Z$ with the non-zero brackets given by 
$$[X_j,Y_k]=\delta_{jk}Z.$$
We have that $\mathfrak{h}_{2m+1}$ is of step $2$ with $(\mathfrak{h}_{2m+1})_{-1}=\sum_j \R X_j +\R Y_j$ and $(\mathfrak{h}_{2m+1})_{-2}=\mathfrak{z}=\R Z$. We now apply a construction of Mohsen \cite{mohsen2}, later expanded on in \cite{goffkuz}. The Mohsen modification of $\mathfrak{h}_{2m+1}$ is as a vector space defined by 
$$\tilde{\mathfrak{h}}_{2m+1}=\mathfrak{h}_{2m+1}\oplus (\mathfrak{h}_{2m+1})^*\oplus \R Z_0,$$
and is generated by $4m+3$ generators $\tilde{X}_1,\ldots, \tilde{X}_m, \tilde{Y}_1,\ldots, \tilde{Y}_m,\tilde{Z}$, $e_{X_1},\ldots, e_{X_m}, e_{Y_1},\ldots, e_{Y_m}$, $e_{Z},Z_0$ subject to the bracket relations
$$\begin{cases}
[\tilde{X}_j,\tilde{Y}_k]=\delta_{jk}\tilde{Z},\\
[\tilde{X}_j,e_Z]=e_{\tilde{Y}_j},\\
[\tilde{Y}_j,e_Z]=-e_{\tilde{X}_j},\\
[\tilde{Z},e_Z]=Z_0,\\
[\tilde{X}_j,e_{X_k}]=[\tilde{Y}_j,e_{Y_k}]=\delta_{j,k}Z_0
\end{cases}$$
We have that $(\tilde{\mathfrak{h}}_{2m+1})_{-1}$ is spanned by $\tilde{X}_1,\ldots, \tilde{X}_m, \tilde{Y}_1,\ldots, \tilde{Y}_m,e_{Z}$, $(\tilde{\mathfrak{h}}_{2m+1})_{-2}$ is spanned by $\tilde{Z}, e_{X_1},\ldots, e_{X_m}, e_{Y_1},\ldots, e_{Y_m}$ and finally $(\tilde{\mathfrak{h}}_{2m+1})_{-3}=\mathfrak{z}$ is spanned by $Z_0$. 

The Mohsen modification always produces a nilpotent Lie algebra with one-dimensional center and flat orbits, so $\Gamma=\R^\times$. An interesting feature of the Mohsen modification is that a representation in a flat orbit $\hbar\in \R^\times$ is an extension of the left regular representation. In light of this fact, we can readily describe $\pi_\hbar(D_\gamma)$ where $\hbar$ ranges over the flat orbits $\Gamma\cong \R^\times$. Here 
$$D_\gamma=\sum_{j=1}^m(\tilde{X}_j^2+\tilde{Y}_j^2)+e_{Z}^2+\gamma_0\tilde{Z}+\sum_{l=1}^m (\gamma_{2l-1}e_{X_l}+\gamma_{2l}e_{Y_l}).$$ 
The representation in the flat orbit in particular has functional dimension $2m+1$, and we write the coordinates in that space by $(x,y,z)\in \R^m\times \R^m\times \R$. By homogeneity, we compute that 
\begin{align*}
\pi_\hbar(D_\gamma)=|\hbar|\sum_{j=1}^m&\left((\partial_{x_j}-\frac{y}{2}\partial_z)^2+(\partial_{y_j}+\frac{x}{2}\partial_z)^2\right)-|\hbar|z^2+\\
&+\hbar \left(\gamma_0i\partial_z-\frac{1}{2}\sum_{l=1}^m (\gamma_{2l-1}x_l+\gamma_{2l}y_l)\right).
\end{align*}
So, up to unitary equivalence, $\pi_\hbar(D_\gamma)$ coincides with the operator 
$$|\hbar|(-L_B+\partial_B^2)+i\hbar\left(\gamma_0 B-\frac{1}{2}\sum_{l=1}^m (\gamma_{2l-1}x_l+\gamma_{2l}y_l)\right),$$
acting on $L^2(\R^m_x\times \R^m_y\times \R_B)$. Here $L_B$ denotes the Landau hamiltonian at magnetic field strength $B$.

\begin{prop}
The subset $\mathcal{E}(\tilde{\mathfrak{h}}_{2m+1},N)\subseteq M_N(\C)^{2m+1}$ consists of the matrices $\gamma=(\gamma_{l})_{l=0}^{2m}\in M_N(\C)^{2m+1}$ such that the two operators
$$-L_B+\partial_B^2\pm i\left(\gamma_0 B-\frac{1}{2}\sum_{l=1}^m (\gamma_{2l-1}x_l+\gamma_{2l}y_l)\right) $$
 on $L^2(\R^m_x\times \R^m_y\times \R_B)$, are invertible.
\end{prop}

\begin{proof}
The Lie algebra $\tilde{\mathfrak{h}}_{2m+1}$ has flat coadjoint orbits and one-dimensional centre, so $D_\gamma$ is $H$-elliptic if and only if $\pi_\hbar(D_0)\pi_\hbar(D_\gamma)^{-1}$ exists and is uniformly bounded in $\hbar\in \Gamma=\mathfrak{z}^*\setminus \{0\}$, see \cite[Theorem 22.20]{goffkuz}. It is clear from homogeneity that $\pi_\hbar(D_0)\pi_\hbar(D_\gamma)^{-1}$ exists and is uniformly bounded if and only $\pi_\hbar(D_\gamma)$ is invertible for $\hbar=\pm 1$ and the proposition follows.
\end{proof}

\section{A final detail in a proof of Rothschild-Stein}
\label{fillingap}

 Theorem 2 in \cite{rothstein} (discussed above in Section \ref{seconestan}) can nowadays be proved more simply by using the Rockland criterion -- that was not formulated at the time of \cite{rothstein} but is now a theorem, cf. \cite{AMY,HN79a,HNo,melinoldpreprint,vanerpyuncken}. Even if \cite[Theorem 2]{rothstein} is correct as stated, the proof contains a minor oversight: in \cite[equation (4.1) on page 258]{rothstein}, they introduce $2l\leq n$ as the rank of a functional $\rho^*$ while in the next paragraph, the authors construct a representation $\pi_\lambda$ while implicitly assuming that $2l<n$. We discuss the proof here in a way where the case $2l=n$ is also covered. In this remaining case $2l=n$, \cite[Theorem 2]{rothstein} follows from applying the mean value theorem. As one can see from the statement at hand, it is a rank 2 statement and we consider the operator
$$
P = \sum_{j=1}^n X_j^2 + i \sum_{l=1}^m b_l Y_l
$$
where $[\mathfrak g_{-1},\mathfrak g_{-1}] =\mathfrak g_{-2}$, and $b_l\in \mathbb R$.

As observed by Helffer in \cite{H-0} and Beals \cite{Be}, one can use the results of Boutet-Grigis-Helffer giving equivalent conditions to Rockland's criterion.
The conditions of hypoellipticity can be expressed in the following way. For each $\eta \in \mathfrak g_{-2}^*\setminus \{0\}$, let $r_\eta$ denote the rank of the Kirillov form on $\mathfrak g_{-1}\times \mathfrak g_{-1}$. The Kirillov form $\omega_\eta$ is defined by
$$
\omega_\eta(X,X') = \eta([X,X'])
$$
Our assumptions imply that
$$
2 \leq r_\eta \leq \dim \mathfrak g_{-1}\,.
$$
By homogeneity, it is enough to consider the case when $|\eta|=1$. In the notation of \cite{rothstein}, $r_\eta=2l$ and $n=\dim\mathfrak{g}_{-1}$.

Without entering into the details, we can with each $\eta$ associate the value $\lambda (\eta) >0$ as the bottom of the spectrum of some Harmonic Oscillator $H_\eta$. Write $(\lambda_j(\eta))_{j=1}^\infty$ for the growing sequence of eigenvalues of the harmonic oscillator $H_\eta$, so $\lambda(\eta)=\lambda_1(\eta)$. The criterion for hypoellipticity, which is a necessary and sufficient condition, reads
\begin{itemize}
\item If $r_\eta < \dim \mathfrak g_{-1}$, $\sum_{l=1}^m b_l \eta_l \not\in [\lambda(\eta), +\infty)$
\item If $r_\eta= \dim \mathfrak g_{-1}$, $\sum_{l=1}^m b_l\eta_l \not\in \{\lambda(\eta), \lambda_2(\eta),\lambda_3(\eta),\dots\}$. 
\end{itemize}

Note also that $\lambda(\eta)=\lambda(-\eta)$. We now rephrase these conditions under various assumptions. The first result is the following. 

\begin{prop}\label{propa}
If $\sum_{l=1}^m b_l\eta_l < \lambda(\eta)$, for any $|\eta|=1$, the operator $P$ is hypoelliptic.
\end{prop}

This corresponds to \cite[Theorem 1']{rothstein}. To determine if the condition is necessary, we obtain the following two propositions from the discussion above.

\begin{prop}
\label{propb}
If for some $\eta$ (with $|\eta|=1$) we have $\sum_{l=1}^m b_l\eta_l = \lambda(\eta)$, the operator is not hypoelliptic.
\end{prop}

\begin{prop}
\label{propc}
If for some $\eta$ (with $|\eta|=1$) we have $\sum_{l=1}^m b_l \eta_l \geq \lambda(\eta)$, and $r_\eta < \dim \mathfrak{g}_{-1}$, the operator is not hypoelliptic. 
\end{prop}

Since $r_\eta$ is even, this proposition in particular can be applied when $\dim \mathfrak g_{-1}$ is odd. This is the case for Example \ref{n4ex} with $ \mathfrak g = \mathfrak n (4)$. Finally, we have

\begin{prop}
\label{propd}
We assume $\dim \mathfrak g_{-2} \geq 2$. If for some $\eta$ (with $|\eta|=1$) we have $\sum_{l=1}^m b_l \eta_l \geq \lambda(\eta)$, and $\dim \mathfrak g_{-2} \geq 2$, the operator is not hypoelliptic. 
\end{prop}

\begin{proof}
Due to the previous propositions, it remains to treat the case when $\sum_{l=1}^m b_l \eta_l> \lambda(\eta)$ and $r_\eta =\dim \mathfrak g_{-1}$. Since $\dim\mathfrak{g}_{-2} \geq 2$, there exists an $\eta'$ with $|\eta'|=1$ such that
$$
\sum_{l=1}^m b_l \eta'_l=0\,.
$$
By the mean value theorem on the connected sphere of dimension $\dim\mathfrak{g}_{-2}-1$, there exists an $\eta''$ such that $|\eta''|=1$ and $ \sum_{l=1}^m b_l \eta''_l =\lambda(\eta'')$. We can then apply Proposition \ref{propb}.
\end{proof}
It remains to consider the case when $\dim \mathfrak g_{-2}=1$. The case when $r_\eta < \dim \mathfrak g_{-1}$ is treated in Proposition \ref{propc} and the case when $r_\eta = \dim \mathfrak g_{-1}$ corresponds to the Heisenberg algebra.

\vspace{-1mm}
\end{document}